\documentclass[reqno, 12pt]{smfart}

\usepackage{a4wide}

\usepackage{amsfonts, amsthm, amsmath, amssymb}

\RequirePackage{mathrsfs} \let\mathcal\mathscr

\numberwithin{equation}{section}

\newtheorem{theorem}{Th\'eor\`eme}[section]
\newtheorem{lemma}[theorem]{Lemme}
\newtheorem{cor}[theorem]{Corollaire}

\theoremstyle{definition}
\newtheorem*{ack}{Remerciements}

\renewcommand{\phi}{\varphi}

\newcommand{\PP}{\mathbb{P}}
\renewcommand{\AA}{\mathbb{A}}

\newcommand{\ZZ}{\mathbb{Z}}

\newcommand{\NN}{\mathbb{N}}
\newcommand{\QQ}{\mathbb{Q}}

\renewcommand{\leq}{\leqslant}

\renewcommand{\geq}{\geqslant}

\renewcommand{\bar}{\overline}

\newcommand{\x}{\mathbf{x}}

\newcommand{\q}{\mathbf{q}}

\renewcommand{\b}{\mathbf{b}}
\renewcommand{\a}{\mathbf{a}}
\renewcommand{\r}{\mathbf{r}}
\renewcommand{\d}{\mathbf{d}}

\newcommand{\ve}{\varepsilon}

\newcommand{\bve}{\boldsymbol{\varepsilon}}

\DeclareMathOperator{\Pic}{Pic}
\DeclareMathOperator{\Gal}{Gal}

\DeclareMathOperator{\Br}{Br}

\DeclareMathOperator{\Mod}{mod}
\renewcommand{\bmod}[1]{\,(\Mod{#1})}
\renewcommand{\rho}{\varrho}

\newcommand{\Nglob}{N_{\mathrm{glob}}}
\newcommand{\NBr}{N_{\mathrm{Br}}}

\newcommand{\8}{\bmod{8}}
\newcommand{\2}{\bmod{2}}
\renewcommand{\=}{\equiv}

\begin{document}
\title[
Contre-exemples au principe de Hasse
]{Contre-exemples au principe de Hasse
pour certains tores coflasques}

\begin{abstract} Nous \'etudions le comportement asymptotique du 
nombre de vari\'et\'es dans une certaine classe ne satisfaisant pas 
le principe de Hasse. Cette \'etude repose sur des r\'esultats 
r\'ecemment obtenus par Colliot-Th\'el\`ene \cite{web}.
\end{abstract}

\author{R.\ de la Bret\`eche}
\address{
Institut de Math\'ematiques de Jussieu\\
Universit\'e  Denis Diderot\\
Case Postale 7012\\
F-75251 Paris CEDEX 13\\ France}
\email{breteche@math.jussieu.fr}

\author{T.D.\ Browning}
\address{School of Mathematics\\
University of Bristol\\ Bristol\\ BS8 1TW\\ United Kingdom}
\email{t.d.browning@bristol.ac.uk}


\maketitle

\section{Introduction}

Nous nous int\'eressons \`a la fr\'equence de contre-exemples au 
principe de Hasse dans une famille  de vari\'et\'es
alg\'ebriques d\'efinies sur $\QQ$.
Les courbes de degr\'e  $3$ dans $\PP_\QQ^2$
sont l'objet du travail de Bhargava \cite{msri}. Le cas des surfaces 
de Ch\^atelet a \'et\'e r\'ecemment \'etudi\'e par La  Bret\`eche et 
Browning \cite{chat}.

Le but de cet article est de faire de m\^eme pour les vari\'et\'es
affines $Y\subset \AA_\QQ^6$,  d\'efinies par
\begin{equation}\label{eq:Y}
(x^2-ay^2)(z^2-bt^2)(u^2-abw^2)=c,
\end{equation}
avec $a,b,c\in \QQ^*$.  L'arithm\'etique de $Y$ a \'et\'e
\'etudi\'ee par Colliot-Th\'el\`ene \cite[\S 5]{web}, qui a notamment montr\'e
que le  choix de coefficients
$(a,b,c)=(13,17,5)$ donne un contre-exemple au principe de Hasse.
Notre investigation quantitative
est fond\'ee sur son travail.

La vari\'et\'e $Y$ est un espace principal homog\`ene du tore coflasque
$$
(x^2-ay^2)(z^2-bt^2)(u^2-abw^2)=1.
$$
D'apr\`es un r\'esultat de  Sansuc \cite[Cor.~8.7]{sansuc}, l'obstruction
Brauer--Manin est la seule obstruction au principe de Hasse.
Soit $Y^c$ une $\QQ$-compactification lisse de $Y$ et
$\bar{Y^c}=Y^c\times_{\QQ} \bar{\QQ}.$
Une caract\'eristique int\'eressante de $Y$ est le fait qu'il existe 
un  g\'en\'erateur
universel explicite pour le groupe de Brauer
$\Br(Y^c)/\Br(\QQ)=H^1(\Gal(\bar\QQ/\QQ), \Pic(\bar{Y^c}))$.

En fait,
suite \`a  \cite[Thm.~4.1]{web},
si  $a,b,ab\in \QQ^*\smallsetminus {\QQ^*}^2$
on a $\Br(Y^c)/\Br(\QQ)=\ZZ/2\ZZ$ avec
l'alg\'ebre de quarternions
$(x^2-ay^2, b)\in
\Br(\QQ(Y))$ comme g\'en\'erateur,
tandis que, si
l'un des $a,b,ab$ est dans
${\QQ^*}^2$, $Y$ est $\QQ$-rationnelle, et donc
$\Br(Y^c)/\Br(\QQ)=0$.
Nous  utilisons cette description explicite pour
d\'eterminer la fr\'equence \`a laquelle
il existe des contre-exemples au principe de Hasse pour les
vari\'et\'es \eqref{eq:Y}.

Nous  param\'etrons les  vari\'et\'es $Y$ par  l'ensemble
\begin{equation}\label{eq:S}
S=\{(a,b,c)\in (\ZZ\smallsetminus \{0\})^3:  \mbox{$a,b,c$ sans 
facteur carr\'e et $c>0$}\}.
\end{equation}
Il est \'evident  que toute $\QQ$-vari\'et\'e
\eqref{eq:Y} est
$\QQ$-isomorphe \`a la vari\'et\'e d\'efinie par la m\^eme \'equation 
avec $(a,b,c)\in S$.
Notre int\'er\^et principal est de d\'eterminer la r\'epartition des 
\'el\'ements de $S$ tels que $Y(\QQ)$ est vide (ou non-vide).
  \`A la lumi\`ere de  \cite[Prop.~5.1(a)]{web}, pour toute place $v$ 
de $\QQ$ et chaque  $(a,b,c)\in S$, on a
$Y(\QQ_v)\neq \emptyset$.
Il n'y a donc jamais d'obstruction locale pour l'existence de $\QQ$-points.

Soit $S(P)=\{(a,b,c)\in S: \max\{|a|,|b|,c\}\leq P\}$,
pour $P\geq 1$. Nous estimons asymptotiquement,
lorsque $P$ tend vers l'infini, le cardinal
$$
\NBr (P)=
\#\{(a,b,c)\in S(P): Y(\QQ) = \emptyset\}.
$$
Notre r\'esultat principal est le suivant.

\begin{theorem}\label{main}
Lorsque $P\geq 2$, on a
$$
\NBr (P)=
\frac{\tau_1P^3}{\log P}
-
\frac{\tau_2P^3}{(\log P)^{\frac{3}{2}}}
   +O \left(\frac{P^3}{(\log P)^{2}} \right),
$$
o\`u
\begin{align*}
\tau_1&=
\frac{45 }{  \pi^5} \prod_{p>2}\Big( 1+\frac{1}{
2p(p+1)}\Big)+\frac{15 }{ \pi^5} \prod_{\substack{p>2}}\Big(
1+\frac{(-1)^{\frac{p-1}{2}}}{ 2p( p+1) }\Big), \\
   \tau_2
   &=\frac{153}{ 16\pi^{\frac{7}{2}}} \prod_{p}   \frac{(1- \frac{1}{ p
})^{\frac{1}{2}}}{(1+ \frac{1}{ p })}
\Big( 1+\frac{3}{ 2p }+\frac{1}{  p^2 }\Big)
.\end{align*}
\end{theorem}

La diff\'erence
$\Nglob (P)=\#S(P)-\NBr(P)$ est le nombre de vari\'et\'es $Y$ param\'etr\'ees
par~$S(P)$ pour lesquelles   $Y(\QQ)\neq \emptyset$. Le cardinal 
$\#S(P)$ \'etant facile  \`a estimer, nous obtenons le r\'esultat 
suivant.

\begin{cor}
Lorsque $P\geq 2$, on a
\begin{align*}
\Nglob(P)
&= \#S(P) +O\left(\frac{P^3}{\log P}\right)\\
&= \frac{864}{\pi^6}P^3 +O\left(\frac{P^3}{\log P}\right).
\end{align*}
En particulier, on a une proportion asymptotique  de
  100\% des vari\'et\'es Y qui ont des $\QQ$-points.
\end{cor}

\begin{ack}
Pendant l'\'elaboration de cet article, le premier auteur a \'et\'e 
soutenu par un
{\em IUF junior} et le {\em projet ANR (PEPR)}, tandis que le second
auteur a \'et\'e soutenu par la bourse {\em ERC 306457}.
\end{ack}

\section{L'obstruction  Brauer--Manin}

Nous rappelons quelques points cl\'es
du travail de Colliot-Th\'el\`ene
\cite{web}, sur les vari\'et\'es $Y$ d\'efinies en  \eqref{eq:Y}, lorsque
$(a,b,c)$
appartient \`a l'ensemble $S$ d\'efini en \eqref{eq:S}.

Selon  \cite[Prop.~5.1(c)]{web}, on a $Y(\QQ)\neq\emptyset$
s'il existe un nombre premier $p$ tel qu'aucun  des~$a,b, ab$ ne soit
pas un carr\'e dans $\QQ_p^*$.
Supposons que pour chaque premier $p$
l'un au moins des $a,b$ ou $ab$ est un carr\'e
dans  $\QQ_p^*$, alors il d\'ecoule de
\cite[Prop.~5.1(d)]{web} que $Y(\QQ)=\emptyset$ si, et seulement si,
$$
\sum_{\substack{p\\ a\not\in {\QQ_p^*}^2}} [c,b]_p\equiv 1 \bmod 2.
$$
Ici  $[\cdot,\cdot]_p:\QQ_p^*\times \QQ_p^*\rightarrow \ZZ/2\ZZ$ est 
d\'efini par
$(\cdot,\cdot)_p=(-1)^{[\cdot,\cdot]_p}$, o\`u $(\cdot,\cdot)_p$ est
le symbole de   Hilbert.

Lorsque $(a,b,c)\in \ZZ^3$, nous consid\'erons
$$
f(a,b)=\begin{cases}1, &\mbox{si
$a,b$ ou $ab$ est dans  ${\QQ_p^*}^2$ pour tout $p$,}\\
0, & \mbox{sinon},
\end{cases}
$$
et
\begin{equation}\label{defh}
h(a,b,c)=\prod_{\substack{p\\ a\not\in {\QQ_p^*}^2}} (c,b)_p.
\end{equation}
Notre probl\`eme est donc d'\'evaluer, lorsque $P$ tend vers
l'infini, la quantit\'e
\begin{equation}\label{eq:NBr}
\begin{split}
\NBr(P)
&=\sum_{(a,b,c)\in S(P)} \frac{f(a,b)}{2} \big(1-h(a,b,c)\big)
\\
&=
\frac{1}{2}\left(N_1(P)-N_2(P)\right),
\end{split}
\end{equation}
o\`u les d\'efinitions de
$N_1(P)$ et $N_2(P)$
sont \'evidentes.
Notre analyse de
$N_1(P)$ et $N_2(P)$
est inspir\'ee du travail  de Friedlander et Iwaniec \cite{fi}.

Nous commen\c cons avec l'observation
$$
N_1(P)
=\sum_{(a,b,c)\in S(P)} f(a,b)
=\left(\sum_{
\substack{
(a,b)\in \ZZ^2\\
|a|,|b|\leq P
}}  \mu^2(a)\mu^2(b)f(a,b)\right) \left(\sum_{1\leq c\leq P}\mu^2(c)\right),
$$
o\`u $\mu$ est la fonction de M\"obius.
Nous \'etendons la d\'efinition de la fonction $\mu$ de telle sorte que
$\mu(0)=0$.
Le deuxi\`eme facteur est facile \`a estimer. Il vient
\begin{equation}\label{eq:N1}
N_1(P)
=\frac{6P}{\pi^2}
\sum_{
\substack{
(a,b)\in \ZZ^2\\
|a|,|b|\leq P
}}  \mu^2(a)\mu^2(b)f(a,b)
+O(P^{\frac{5}{2}}).
\end{equation}

Il est clair que l'un au moins
des $a,b$ ou $ab$ est un carr\'e dans ${\QQ_p^*}$ pour chaque premier 
$p\nmid 2ab$.
Quand  $p=2$ et  $ab$ est impair, la condition relative \`a $p=2$ 
contenue dans $f(a,b)$
est que l'un au moins des
$a,b$ or $ab$ est congru \`a $1$ modulo $8$.
Rappelant que $a,b$ sont des entiers sans facteur carr\'e, nous avons 
alors l'\'egalit\'e
\begin{equation}\label{f=gprodp}\begin{split}
f(a,b)=~&
f_2(a,b)\prod_{\substack{p\mid a\\ p\nmid 2b}}
\frac{1}{2}\left(1+\left(\frac{b}{p}\right)\right)
\prod_{\substack{p\mid b\\ p\nmid 2a}}
\frac{1}{2}\left(1+\left(\frac{a}{p}\right)\right)
\\&\times \prod_{\substack{p\mid  \gcd(a,b)\\ p\neq 2}}
\frac{1}{2}\left(1+\left(\frac{ab/(a,b)^2}{p}\right)\right),\end{split}
\end{equation}
avec
$$
f_2(a,b)=\begin{cases}
1, & \mbox{si  $2\nmid ab$ et $1 \in \{a,b, ab\}\8$,}\\
1, & \mbox{si  $2\mid a$ et $b\equiv 1\8$,}\\
1, & \mbox{si  $2\mid b$ et $a\equiv 1\8$,}\\
1, & \mbox{si  $2\mid (a,b)$,   $ab\equiv 4\bmod {32}$,}\\
0, & \mbox{sinon}.
\end{cases}
$$

   Avec $\vartheta$ d\'efini sur les impairs par $\vartheta(k)=1$ si
$k\equiv 3\bmod 4$ et $\vartheta(k)=0$ si $k\equiv 1\bmod 4$, la loi
de r\'eciprocit\'e quadratique s'\'enonce, lorsque $k$ et $\ell$ sont 
des nombres entiers impairs  premiers entre eux, sous la forme
\begin{equation}\label{loiquad}
\left(\frac{k}{\ell}\right)\left(\frac{\ell}{k}\right)=(-1)^{\vartheta(k)\vartheta(\ell)}.
\end{equation}
Nous terminons cette section par quelques mots sur les symboles  de Hilbert
  (cf.\ \cite[\S 3]{serre}).
Soit $p$ un nombre premier et  $x,y\in \QQ_p^*$.
Supposant que $x=p^{\xi}u$ et $y=p^{\eta}v$, o\`u $u,v $ sont des 
$p$-unit\'es, alors pour tout  $p>2$ nous avons
\begin{equation}\label{eq:hilbert-p}
(x,y)_p=(-1)^{\xi\eta\vartheta(p)} \left(\frac{u}{p}\right)^\eta
\left(\frac{v}{p}\right)^\xi,
\end{equation}
tandis que lorsque $p=2$, nous avons
\begin{equation}\label{eq:hilbert-2}
(x,y)_2=(-1)^{ \vartheta(u )\vartheta(v )  +\frac{\xi(v^2-1)}{8}
+\frac{\eta(u^2-1)}{8}}.
\end{equation}

\section{Lemmes techniques}\label{s:lemmes}

Lorsque $\nu\in \ZZ_{\geq 0}$, nous consid\'erons
$$
\phi_\nu(r)=\prod_{p\mid r} \Big(1+\frac{\nu}{ 2
p}\Big)^{-1}.
$$
Nous aurons besoin de \cite[Cor.~2]{fi}, concernant
$$
C_\nu(x;a,d,r,q,\chi)=\sum_{\substack{
n\leq x\\
   (n ,d)=1\\
n\equiv a\bmod {r}
}} \frac{ \mu^2(n)\chi(n)  }{2^{\omega(n)}}\phi_\nu(n),
$$
o\`u $\chi$ est un caract\`ere modulo $q$ et $\omega(n)$ d\'esigne le 
nombre de facteurs premiers distincts de~$n$.
Posons
$$
c_\nu(r)=
c_\nu(1)\prod_{p\mid r } \Big(1+\frac{\phi_\nu(p)}{2p}\Big)^{-1} ,
$$
avec
$$
c_\nu(1)=\frac{1}{\sqrt{\pi}}\prod_{p }
\Big(1+\frac{\phi_\nu(p)}{2p}\Big)\Big(1-\frac{1}{p}\Big)^{\frac{1}{2}}.
$$
Notons aussi  $\delta$ la fonction caract\'eristique des caract\`eres
principaux.

\begin{lemma}\label{l:ficor2} Soient $A>0$ fix\'e et $\nu\in \{0,1,2\}.$
Lorsque $a,$ $d,$ $r$, $q$ sont des entiers satisfaisant $(a ,r)=(
d,rq)=(r,q)=1$, $x\geq 2$ et $\chi$ est un caract\`ere modulo $q$,
on a
\begin{align*}
C_\nu(x;a,d,r,q,\chi)
&=  \delta(\chi) \frac{c_\nu(drq)}{\phi(r)}\frac{x}{\sqrt{\log
x}}\Big\{1+O\left(\frac{(\log 3drq)^{\frac{3}{2}}}{\log x}\right)\Big\}
+  O\left(\frac{\tau(d)rqx}{(\log x)^{A}}\right).
\end{align*}
\end{lemma}

\begin{proof}
Dans \cite[Cor.~2]{fi}, ce r\'esultat est d\'emontr\'e pour $\nu=0$, 
la condition
suppl\'ementaire $(a,q)=1$ \'etant inutile. Les cas $\nu=1,2$ se
d\'emontrent de la m\^eme mani\`ere.
\end{proof}

Soient $\x=(x_1,x_2)$, $\a=(a_1,a_2)$, $\d=(d_1,d_2)$, $\r=(r_1,r_2)$ 
et $\q=(q_1,q_2)$. Du Lemme~\ref{l:ficor2}, nous d\'eduisons 
l'estimation de la somme
$$
Q(\x;\a,\d,r,\q,\chi_1,\chi_2)=
\sum_{\substack{(n_1,n_2)\in \NN^2\\
n_i\leq x_i\\
   (n_i,d_i)=1, (n_1,n_2)=1\\
n_i\equiv a_i\bmod {r}
}}   \mu^2(n_1n_2)  \frac{  \chi_1(n_1)\chi_2(n_2)  }{2^{\omega(n_1n_2)}},
$$
o\`u les $\chi_i$
sont des caract\`eres modulo $q_i$.

\begin{cor}\label{c:estQ} Soit $A>0$ fix\'e.
Lorsque $a_1, a_2$, $d_1,$ $d_2,$ $r$, $q_1$, $q_2$ sont des entiers 
satisfaisant $(a_i,r )= ( d_i, q_i)=(r ,d_1d_2q_1q_2)=1$, $x_i\geq 2$ 
et $\chi_i$
sont des caract\`eres modulo $q_i$, on a
\begin{align*}
Q(\x;\a,\d,r,\q,\chi_1,\chi_2)=~&
  \delta(\chi_1) \delta(\chi_2) \frac{c(\d,r)}{\phi(r
)^2}\frac{x_1x_2}{\sqrt{\log x_1\log x_2}}\left\{1+O\left(\frac{(\log
3d_1d_2rq_1q_2)^{\frac{3}{2}}}{\log \min\{ x_1,x_2\}}\right)\right\}\\
&+  O\left(\frac{\tau(d_1d_2)rq_1q_2x_1x_2}{(\log \min\{
x_1,x_2\})^{A}}\right),
\end{align*}
avec
$$
c(\d,r)=
\frac{6}{\pi^3} \frac{ \phi_1(d_1r)
\phi_1(d_2r)}{\phi_1(d_1d_2r)^2}\phi_2(d_1d_2r) .
$$
\end{cor}

\begin{proof}
Notons $Q$ la quantit\'e \`a estimer. Une interversion
de M\"obius fournit
\begin{align*}Q=\!\!\!\!\!\!
\sum_{\substack{n\in\NN\\
(n,d_1d_2rq_1q_2)=1}}\!\!\!\!\frac{\mu(n)\chi_1(n)\chi_2(n)}{4^{\omega(n)}}
C_0\left(\frac{x_1}{n};n^{-1}a_1,d_1n,r 
,q_1,\chi_1\right)C_0\left(\frac{x_2}{n};n^{-1}a_2,d_2n,r
,q_2,\chi_2\right) .
\end{align*}
Nous pouvons restreindre la sommation aux entiers $n\leq T$ o\`u $T=(\log
\min\{ x_1,x_2\})^{\frac{A}{2}},$ la contribution compl\'ementaire \'etant
major\'ee par $O(x_1x_2/T)$.
Nous appliquons ensuite le  Lemme~\ref{l:ficor2} avec $\nu=0$. Il vient
\begin{align*}
Q  =~&
\delta(\chi_1) \delta(\chi_2)
\hspace{-0.4cm}
\sum_{\substack{n\leq T\\
(n,d_1d_2r)=1}}
\hspace{-0.4cm}
\frac{\mu(n)c_0(d_1nr )c_0(d_2nr)x_1x_2}{
4^{\omega(n)}n^2 \phi(q)^2\sqrt{\log x_1\log
x_2}}\left\{1+O\left(\frac{(\log 3d_1d_2 rq_1q_2)^{\frac{3}{2}}}{\log
(\min\{ x_1,x_2\})}\right)\right\}
\\
&+  O\left(\frac{\tau(d_1d_2)rq_1q_2x_1x_2T}{(\log (\min\{
x_1,x_2\}))^{A}}+\frac{x_1x_2}{T}\right)\\
=~&\delta(\chi_1)
\delta(\chi_2)\frac{c(\d,r)x_1x_2}{\phi(q )^2\sqrt{\log
x_1\log x_2}} \left\{1+O\left(\frac{(\log 3d_1d_2
rq_1q_2)^{\frac{3}{2}}}{\log \min\{ x_1,x_2\}}\right)\right\}\\
& +  O\left(\frac{\tau(d_1d_2)rq_1q_2x_1x_2 }{(\log \min\{
x_1,x_2\})^{\frac{A}{2}}} \right).
\end{align*}
Ici nous avons
   \begin{align*}c(\d,r)
   &=c_0(d_1 r ) c_0(d_2 r)
\sum_{\substack{n\in\NN\\
(n,d_1d_2r)=1}}\frac{\mu(n)\phi_1(n)^2}{  4^{\omega(n)}n^2}
\\
&=\frac{1}{\pi}  \phi_1(d_1r) \phi_1(d_2r)
\prod_{p\mid d_1d_2 r} \Big(1-\frac{1}{(2p+1)^2}\Big)^{-1}
\prod_{p}\Big(1-\frac{1}{p^2}\Big)
\\&=\frac{6}{\pi^3}  \phi_1(d_1r)
\phi_1(d_2r)\frac{\phi_2(d_1d_2r)}{\phi_1(d_1d_2r)^2} ,
\end{align*}
ce qui fournit donc le r\'esultat quitte \`a modifier la valeur du
param\`etre $A$.
\end{proof}

Lorsque $q_1q_2$ est grand, le  Corollaire \ref{c:estQ} est inutilisable.
Nous aurons besoin ainsi du r\'esultat suivant
\cite[Lemme~2]{fi}.

\begin{lemma}\label{l:filem2} Soient
$\{\alpha_m\}_{m\in \NN},$ $\{\beta_n\}_{n\in \NN}$ des suites de 
nombres complexes telles que $|\alpha_m|,|\beta_n|\leq 1$, dont le 
support est inclus dans les nombres impairs.
Lorsque $M,N\geq 1$, on a
\begin{align*}
\sum_{m\leq M}\sum_{n\leq N} \alpha_m\beta_n\left(\frac{n}{m}\right)\ll
(MN^{\frac{5}{6}
}+NM^{\frac{5}{6}
})(\log 3 MN)^{\frac{7}{6}
}.
\end{align*}
\end{lemma}

\section{\'Etude de  $N_1(P)$}

Pour estimer  $N_1(P)$ \`a partir de \eqref{eq:N1}, nous consid\'erons
$$T(P)=\sum_{
\substack{
(a,b)\in \ZZ^2\\
|a|,|b|\leq P
}}  \mu^2(a)\mu^2(b)f(a,b) .
$$
Lorsque $\bve=(\varepsilon_1,\varepsilon_2)\in \{\pm 1\}^2$
et $\alpha,\beta\in \{0,1\}$, nous
  notons
  $T(P;  \alpha,\beta,\bve)$ la contribution dans
$T(P)$ des couples $(a,b)$ tels que $2^\alpha\parallel a$ et
$2^\beta\parallel b$,
$\varepsilon_1a>0, $ $ \varepsilon_2b>0$.
Les couples $(a2^{-v_2(a)},b2^{-v_2(b)})$ appartiennent \`a un
ensemble $E_{\alpha,\beta}$ modulo $8$,
avec
\begin{equation}\label{eq:set-E}
\begin{split}
E_{0,0}&=\{(1,\pm 1), (1,\pm 3), (-1,\pm 1), (\pm 3,1),  (\pm 3,\pm 3)\},\\
E_{1,0}&=\{(\pm 1,1), (\pm 3,1)\},\\
E_{0,1}&=\{(1,\pm 1), (1,\pm 3)\},\\
E_{1,1}&=\{(\pm 1,\pm1), (\pm 3,\pm 3)\}.
\end{split}
\end{equation}
Lorsque $\bve\in \{\pm 1\}^2$ et
     $(a2^{-\alpha},b2^{-\beta})\equiv (a_0,b_0)\bmod 8 $ avec
$(a_0,b_0)\in E_{\alpha,\beta}$, $\varepsilon_1a>0, $ et $
\varepsilon_2b>0$, la
formule \eqref{f=gprodp} s'\'ecrit aussi
\begin{equation}\label{calculfab}\mu^2(a)\mu^2(b)f(a,b) =   \sum_{
\substack{
(k,k',\ell,\ell',m,m')\in  \NN^6\\
a=\varepsilon_12^{\alpha} kk'mm'\\b=\varepsilon_22^\beta\ell \ell'mm'
}}\!\!\!
   \frac{ \mu^2(2kk'mm' \ell \ell' ) }{2^{\omega(kk' \ell \ell'mm')}}
\left(\frac{b}{k}\right)\left(\frac{a}{\ell}\right)\left(\frac{ab/(a,b)^2}{m}\right)
.
\end{equation}
Il vient
\begin{align*}
T(P; \alpha,\beta,\bve)
   =~&\sum_{(a_0,b_0)\in E_{\alpha,\beta}}\sum_{\substack{
(k,k',\ell,\ell',m,m')\in \NN^6\\
2^{\alpha} kk'mm',2^\beta\ell \ell'mm'\leq P\\
(\varepsilon_1kk'mm',\varepsilon_2\ell \ell'mm')\equiv (a_0,b_0)\bmod 8
}} \frac{ \mu^2(2kk'mm' \ell \ell' ) }{2^{\omega(kk' \ell \ell'mm')}}
\\
&\qquad\times    \left(\frac{\varepsilon_22^\beta\ell
\ell'mm'}{k}\right)\left(\frac{\varepsilon_1 2^{\alpha}
kk'mm'}{\ell}\right)\left(\frac{\varepsilon_1\varepsilon_22^{\alpha+\beta}
kk' \ell \ell'}{m}\right).
\end{align*}
La loi de r\'eciprocit\'e quadratique \eqref{loiquad} et la
multiplicativit\'e des caract\`eres fournissent
\begin{equation}\label{appliloi}\left(\frac{\varepsilon_22^{\beta}
\ell \ell'mm'}{k}\right)\!\!\left(\frac{\varepsilon_1 2^{\alpha}
kk'mm'}{\ell}\right)\!\!\left(\frac{\varepsilon_1\varepsilon_22^{\alpha+\beta}
kk' \ell \ell'}{m}\right)
=u
\left(\frac{  \ell' }{k }\right)\!\!\left(\frac{ k' }{\ell
}\right)\!\!\left(\frac{  \ell'k'  }{ m}\right) \left(\frac{m'}{k
\ell }\right),
\end{equation}
avec
\begin{equation}\label{defu}u=u(k,\ell,m)
=(-1)^{\vartheta(k)\vartheta(\ell)+\vartheta(k)\vartheta(m)+\vartheta(m)\vartheta(\ell)}
   \left(\frac{\varepsilon_22^{\beta}
}{km}\right)\left(\frac{\varepsilon_1 2^{\alpha}  }{\ell m}\right)
.\end{equation}

Un calcul simple fournit le r\'esultat suivant.

\begin{lemma}\label{lem:calculu}
  Lorsque $(k,\ell)\=(k_0,\ell_0)\bmod 8$  et $(\varepsilon_1
k_0 m,\varepsilon_2  \ell_0 m)\in E_{\alpha,\beta} $, on a
$$
u(k,\ell,m)=(-1)^{\vartheta(\varepsilon_1k_0m)\vartheta(\varepsilon_2\ell_0 m)+
   \vartheta(\varepsilon_1) \vartheta(\varepsilon_2)+\vartheta(m)}
.
$$
\end{lemma}

\begin{proof}
Nous avons toujours
$$
   \left(\frac{ 2^{\beta}  }{k_0m}\right)\left(\frac{  2^{\alpha}
}{\ell_0m}\right) =1.
$$
En effet,
le cas $(\alpha,\beta)=(0,0)$ \'etant trivial, regardons le cas
$(\alpha,\beta)=(1,0)$. Il d\'ecoule de \eqref{eq:set-E} que
$\varepsilon_2\ell_0m\equiv 1\bmod 8$  ce qui montre la formule dans
ce cas. Les raisonnements sont identiques pour
$(\alpha,\beta)=(0,1)$ ou $ (1,1)$.  Nous avons donc  bien
l'expression attendue.
Comme
\begin{align*}
 \left(\frac{\varepsilon_2}{km}\right)\left(\frac{\varepsilon_1}{\ell 
m}\right)
    =(-1)^{
\vartheta(k_0m)\vartheta(\varepsilon_2)+\vartheta(\ell_0m)\vartheta(\varepsilon_1)},
\end{align*}
nous avons
\begin{align*}  u(k,\ell,m)&
=(-1)^{\vartheta(k)\vartheta(\ell)+\vartheta(k)\vartheta(m)+\vartheta(m)\vartheta(\ell)+
\vartheta(k_0m)\vartheta(\varepsilon_2)+\vartheta(\ell_0
m)\vartheta(\varepsilon_1)}
\\&
=(-1)^{\vartheta(\varepsilon_1k_0m)\vartheta(\varepsilon_2\ell_0 m)+
   \vartheta(\varepsilon_1) \vartheta(\varepsilon_2)+\vartheta(m)}
    ,
\end{align*}
ce qui fournit le r\'esultat recherch\'e.
\end{proof}

Nous reprenons la d\'emarche d\'evelopp\'ee  dans \cite{fi}.
Pour cela, nous consid\'erons $$V=(\log P)^{B},$$ o\`u $B$ est un
param\`etre qui sera choisi suffisamment grand en fonction de la
valeur de $A$ prise dans les applications du Corollaire \ref{c:estQ}.

La contribution \`a $T(P; \alpha,\beta,\bve)$
des couples d'entiers $(a,b)$ tels que $mm'>V$ est $\ll P^2/V$.
Dor\'enavant, nous nous restreignons au cas $mm'\leq V$.

La contribution \`a $T(P; \alpha,\beta,\bve)$
des couples d'entiers $(a,b)$ tels que $k\leq V$ et $k'\leq V$  est
$\ll PV^2\log P$. De m\^eme lorsque $\ell  \leq V$ et $\ell'\leq V$.

Gr\^ace au
  Lemme
  \ref{l:filem2},
  du fait de  la pr\'esence du facteur  $(\frac{ \ell' }{k
})$,
   la contribution du cas $\ell'> V$ et
$k> V$  est
   \begin{align*}
   &\ll
\sum_{mm'\leq V}  \sum_{\ell \leq P/(mm'V)}\sum_{k'\leq
P/(mm'V)}
\frac{P^2}{mm'k'\ell}\left\{
\left(\frac{P}{mm'\ell}\right)^{-\frac{1}{6}}+\left(\frac{P}{mm'k'}\right)^{-\frac{1}{6}}\right\}(\log 
P)^{\frac{7}{6}}\\
&\ll
P^{2}V^{-\frac{1}{6}} (\log P)^{\frac{13}{6}}(\log V)^2,
\end{align*}
ce qui suffit lorsque $B> 25$. Nous avons la m\^eme majoration 
lorsque $\ell > V$ et $k'> V$ gr\^ace  \`a la
pr\'esence du facteur   $(\frac{k'}{\ell})$.

Il nous reste \`a traiter le cas $k,\ell\leq V$ ou $k',\ell'\leq V$.
Dans le premier cas, nous sommes amen\'es \`a consid\'erer lorsque
$(\varepsilon_1mm'kk'_0,\varepsilon_2mm'\ell\ell'_0)\in
E_{\alpha,\beta}$ la somme
\begin{align*}T_{k,\ell}(k'_0,\ell'_0,m,m')& =
\sum_{
\substack{
( k', \ell' )\in \NN^2\\
   k'\leq P/(2^{\alpha} mm'k),\ell' \leq P/(2^{\beta} mm'\ell)\\( k' , \ell')=
   (k'\ell', mm'k\ell )=1\\
   (k',\ell')\equiv (k'_0,\ell'_0)\bmod 8
}} \frac{ \mu^2(k')   }{2^{\omega(k')}}\frac{ \mu^2(\ell')
}{2^{\omega(\ell')}}\left(\frac{  \ell' }{km }\right)\left(\frac{ k'
}{\ell m }\right) ,\end{align*}
alors que, dans le deuxi\`eme cas, nous estimerons lorsque
$(\varepsilon_1mm' k_0k',\varepsilon_2mm'\ell_0\ell')\in
E_{\alpha,\beta}$ la somme
$$T'_{k',\ell'}(k_0,\ell_0,m,m')=
\sum_{
\substack{
( k , \ell  )\in \NN^2\\
   k \leq P/(2^{\alpha} mm'k'),\ell  \leq P/(2^{\beta} mm'\ell')\\( k  , \ell )=
   (k \ell , mm'k'\ell' )=1\\
   (k,\ell)\equiv (k_0,\ell_0)\bmod 8
}} \frac{ \mu^2(k )   }{2^{\omega(k )}}\frac{ \mu^2(\ell )
}{2^{\omega(\ell )}}\left(\frac{   \ell'm'}{k  }\right)\left(\frac{
k'm'}{\ell  }\right).
$$
En effet $u(k,\ell,m)=u(k_0,\ell_0,m)$, lorsque
$(k,\ell)\=(k_0,\ell_0)\bmod 8$, ne d\'epend pas de $(k,\ell)$.

Nous avons
\begin{align*}
T_{k,\ell}(k'_0,\ell'_0,m,m')&=
Q\left( \frac{P}{2^{\alpha} mm'k},\frac{P}{2^{\beta} 
mm'\ell};k_0',\ell_0', m'k ,
m' \ell,8,\ell m,k m,\chi_{\ell m},\chi_{k m}\right),
\end{align*}
o\`u $\chi_n(\cdot)=\left(\frac{ \cdot  }{n }\right)$.
Cette somme peut donc \^etre estim\'ee gr\^ace au Corollaire \ref{c:estQ}.
Nous obtenons
\begin{align*}
T_{k,\ell}(k'_0,\ell'_0,m,m')
=~&  \frac{ {\bf 1}_{k =\ell =m=1} }{2^{2+\alpha+\beta} \pi^3 }
\frac{\phi_2( m')}{m'^2}
    \frac{P^2}{\log P}\left\{ 1+O\Big(\frac{\log (2m' )}{\log
P}\Big)\right\}
+O\left( \frac{P^2\tau(kl)}{(\log P)^{
A}}\right),
\end{align*}
o\`u nous avons utilis\'e la formule $c(m',m',8,1,1)= {4}\phi_2(m')/{\pi^3}.$

De m\^eme, posant  $$u'(k_0,\ell_0,m') =(-1)^{\vartheta(
\ell_0)\vartheta(k'm')+\vartheta(k_0  )\vartheta(\ell'm') } ,$$
nous  obtenons gr\^ace \`a \eqref{loiquad}
\begin{align*}
T'_{k',\ell'}(k_0,\ell_0,m,m')=~&
u'(k_0,\ell_0,m')T_{k',\ell'}(k_0,\ell_0,m',m)
\\
=~&    \frac{ {\bf 1}_{k '=\ell' =m'=1} }{2^{2+\alpha+\beta} \pi^3 }
\frac{\phi_2( m )}{m^2}\frac{P^2}{\log P}\left\{ 1+O\Big(\frac{\log
(2m)}{\log P}\Big)\right\}
+O\left(
\frac{P^2\tau(k'\ell')}{(\log P)^{ A}}\right).
\end{align*}
Enfin, nous en d\'eduisons
\begin{align*}
T(P,\alpha,\beta,\bve)=~&
\frac{1}{\pi^3}  \Big(\sum_{
\substack{
   m\in\NN \\2\nmid m
}} \frac{ \mu^2( m)  \phi_2(m) }{2^{\omega( m)}m^2}\tau_{\alpha,\beta} (m)
    \Big)\frac{P^2}{\log P}\left\{ 1+O\Big(\frac{1}{\log P}\Big)\right\}\\
& + O\left( P^2 (\log V)^3\left\{\frac{V^3}{(\log P)^{ A}}+
\frac{(\log P)^{\frac{13}{6}}}{V^{\frac{1}{6}}} \right\} \right),
\end{align*}
avec
$$\tau_{\alpha,\beta}(m) =\frac{\#  E_{\alpha,\beta}}{2^{ \alpha+\beta}}+
\sum_{\substack{\varepsilon_1,\varepsilon_2\in \{\pm 1\}\\
(k_0,\ell_0)\in (\ZZ/8\ZZ)^2
\\
(\varepsilon_1 k_0 m,\varepsilon_2  \ell_0 m)\in E_{\alpha,\beta}
\bmod 8}}\!\!\frac{u(k_0,\ell_0,m) }{2^{2+\alpha+\beta}}  ,$$
o\`u $u$ a \'et\'e d\'efini en \eqref{defu}.

Gr\^ace \`a \eqref{eq:set-E}, nous avons
$$
\sum_{(\alpha,\beta)\in \{ 0,1\}^2}\frac{\#  E_{\alpha,\beta}}{2^{
\alpha+\beta}}=15 .
$$
Le Lemme \ref{lem:calculu} et les \'egalit\'es \eqref{eq:set-E} 
fournissent ainsi
\begin{align*}
\sum_{(\alpha,\beta)\in \{
0,1\}^2}\tau_{\alpha,\beta}(m)&=15+\sum_{(\alpha,\beta)\in \{
0,1\}^2}\sum_{\substack{\varepsilon_1,\varepsilon_2\in \{\pm 1\}\\
(u_0  ,v_0  )\in E_{\alpha,\beta}\bmod 8}}\frac{ (-1)^{
\vartheta(\varepsilon_1) \vartheta(\varepsilon_2)}
}{2^{2+\alpha+\beta}}(-1)^{\vartheta(u_0)\vartheta(v_0)+ \vartheta(m)
}  \\
&=15+ (-1)^{\vartheta(m)} \sum_{(\alpha,\beta)\in \{ 0,1\}^2}
\frac{1}{2^{1+\alpha+\beta}}\sum_{\substack{
(u_0  ,v_0  )\in E_{\alpha,\beta}\bmod
8}}(-1)^{\vartheta(u_0)\vartheta(v_0)  }  \\
&
  =15+5(-1)^{\vartheta(m)} .
\end{align*}

Un simple calcul fournit
\begin{align*}\sum_{
\substack{
   m\in\NN \\2\nmid m
}} &\frac{z^{\vartheta(m)}\mu^2( m)\phi_2(m)}{2^{\omega( m)}m^2}
= \prod_{\substack{p>2}}\Big( 1+\frac{z^{\vartheta(p)}}{ 2p( p+1)
}\Big),  \quad (z\in \{\pm1\}).\end{align*}
En choisissant
$A=81$ and $B=26$,
nous obtenons
\begin{align*}T(P)& =C\frac{P^2}{\log P}\Big\{ 1+O\Big(\frac{1}{\log
P}\Big)\Big\}, \end{align*}
o\`u
$$
C=
\frac{15 }{  \pi^3} \prod_{p>2}\Big( 1+\frac{1}{
2p(p+1)}\Big)+\frac{5 }{ \pi^3} \prod_{\substack{p>2}}\Big(
1+\frac{(-1)^{\vartheta(p)}}{ 2p( p+1) }\Big).
$$
\`A partir de \eqref{eq:N1}, il vient
ainsi \begin{equation}\label{estN1P}N_1(P)=
\frac{6C }{  \pi^2} \frac{P^3}{\log P}\Big\{ 1+O\Big(\frac{1}{\log
P}\Big)\Big\}.
\end{equation}

\section{\'Etude de  $N_2(P)$}

Notre objectif dans cette section est d'estimer la somme
$$
N_2(P)
=\sum_{(a,b,c)\in S(P)} f(a,b)h(a,b,c).
$$
Les calculs sont plus compliqu\'es que pour l'estimation de $N_1(P)$
mais rel\`event des m\^emes m\'ethodes.

Lorsque $f(a,b)\neq 0$, nous aurons besoin d'une expression simple de
la fonction $h(a,b,c)$ d\'efinie en \eqref{defh}. Comme  $(c,b)_p=1$
pour tout nombre premier impair  $p$ ne divisant pas $bc$, nous avons
$$h(a,b,c)=\prod_{\substack{p\mid 2bc \\ a\not\in {\QQ_p^*}^2}} (c,b)_p.$$
Rappelons la d\'efinition \eqref{eq:S} de $S$.
Nous  param\'etrons les $(a,b,c)\in S$ par
\begin{equation}\label{parametre}
a =\varepsilon_12^\alpha d_0d_{12}d_{13} a',\quad
b =\varepsilon_22^\beta d_0d_{12}d_{23} b',\quad
c =2^\gamma d_0d_{13}d_{23} c',
\end{equation}
avec $d_{ij}$, $d_0$, $a',$ $b',$ $c'$  des nombres impairs, 
$\alpha,\beta,\gamma \in
\{0,1\}$ et les conditions de coprimalit\'e
\begin{align*}
&(d_{12}, d_{13})=1,\quad (d_{12}, d_{23})=1,\quad (d_{13}, d_{23})=1, \\
&(a'b'c',d_{12}  d_{13} d_{23})=(a',b'c')=(b',c')=1.
\end{align*}

Nous \'ecrivons $h(a,b,c)=h_1(a,b,c)h_2(a,b,c)$, avec
$h_1(a,b,c)$ le produit sur les $p$ impairs et
la quantit\'e $h_2(a,b,c)$ d\'esignant le facteur li\'e \`a $p=2$.
Les deux r\'esultats suivants concernent leur calcul
explicite.

\begin{lemma}\label{lem:h1}
Lorsque  $(a,b,c)\in S$
et $f(a,b)\neq 0$, on a
\begin{align*}
h_1 (a, b,c)&=  \frac{1}{2^{\omega(c')}}
\sum_{nn'n''n''' =c'}\tilde u\left(\frac{ a'  }{nn''}\right)
\left(\frac{ b'  }{n'n''d_0d_{13}}\right),
\end{align*}
o\`u
\begin{align*}
\tilde u&=\tilde u(n,n',n'',n''', d_0, d_{12},d_{13},d_{23} )\\
&=
\mu(n'')
(-1)^{\vartheta(d_0d_{12})\vartheta(d_0d_{13}nn')
}\left(\frac{\varepsilon_12^{\alpha} d_{13} }{nn''}\right)
\left(\frac{\varepsilon_22^{\beta } d_{23} }{n'n''}\right)
\left(\frac{\varepsilon_22^{\beta }   }{d_0d_{13}}\right)
   \left(\frac{2^\gamma  n''n'''}{d_0d_{12}}\right)
     \left(\frac{d_{23}}{d_{12}d_{13}}\right).
\end{align*}
\end{lemma}

\begin{proof}
Lorsque $a,b,c$ sont sans facteur carr\'e et $p\mid a$, alors
$a\not\in {\QQ_p^*}^2$. De plus, lorsque $p\nmid a$ et $p\mid b$, le
fait que $f(a,b)\neq 0$ implique que $a$ est un carr\'e dans
${\QQ_p^*}$ ce qui est exclu. Lorsque $p\nmid a$, nous nous
restreignons \`a  $p\nmid b$ et donc $p\mid c$ et
$(\frac{a}{p})=-1$. Ainsi d'apr\`es \eqref{eq:hilbert-p},
nous avons $(c,b)_p=(\frac{b}{p})$.
Il vient
$$h_1(a,b,c)=\prod_{\substack{p\mid (bc,a) \\ p>2}} (c,b)_p
\prod_{\substack{p\mid  c \\ p\nmid 2ab}}
\frac{1}{2}\Big\{1+\left(\frac{a}{p}\right)+\left(\frac{b}{p}\right)-\left(\frac{ab}{p}\right) 
\Big\}.
$$
En d\'eveloppant le produit, nous obtenons
$$h_1(a,b,c)=\frac{1}{2^{\omega(c')}}\prod_{\substack{p\mid (bc,a) \\
p>2}} (c,b)_p
\sum_{nn'n''n''' =c'}\mu(n'')\left(\frac{a}{n}\right)
\left(\frac{b}{n'}\right) \left(\frac{ab}{n''}\right)  .
$$
De plus, nous avons
\begin{align*}\prod_{\substack{p\mid (bc,a) \\ p>2}} (c,b)_p
&=
\prod_{\substack{p\mid d_0  }}
\left(\frac{-1}{p}\right)\left(\frac{bc/p^2}{p}\right)\prod_{\substack{p\mid
(a,b) \\ p\nmid 2c}}  \left(\frac{c}{p}\right)
\prod_{\substack{p\mid (a,c) \\ p\nmid 2b}}  \left(\frac{b}{p}\right)
\\&=
\left(\frac{-1}{d_0}\right)\left(\frac{bc/d_0^2}{d_0}\right)\prod_{\substack{p\mid 
(a,b) \\ p\nmid 2c}}  \left(\frac{c}{p}\right)
\prod_{\substack{p\mid (a,c) \\ p\nmid 2b}}  \left(\frac{b}{p}\right).
\end{align*}
En utilisant les notations \eqref{parametre}, nous obtenons
\begin{align*} \prod_{\substack{p\mid (bc,a) \\ p>2}} (c,b)_p
    &=
\left(\frac{-1}{d_0}\right)\left(\frac{\varepsilon_22^{\beta+\gamma}
b'c'd_{12}d_{13} }{d_0}\right)  \left(\frac{2^\gamma d_0d_{13}d_{23}
c'}{d_{12}}\right)  \left(\frac{\varepsilon_22^\beta d_0d_{12}d_{23}
b'}{d_{13}}\right)
   \\&=
   \left(\frac{\varepsilon_22^{\beta }  b' }{d_0d_{13}}\right)
\left(\frac{2^\gamma c' }{d_0d_{12}}\right)
   \left(\frac{   d_{23}  }{d_{12}d_{13}}\right)
(-1)^{\vartheta(d_0d_{12})\vartheta(d_0d_{13}) },
\end{align*}
puisque  la loi de r\'eciprocit\'e quadratique \eqref{loiquad} fournit
$$\left(\frac{-1}{d_0}\right)\left(\frac{ d_{12}d_{13} }{d_0}\right)
   \left(\frac{d_0d_{13}}{d_{12}}\right)  \left(\frac{ 
d_0d_{12}}{d_{13}}\right)
=(-1)^{\vartheta(d_0d_{12})\vartheta(d_0d_{13}) }.$$

Cela implique ainsi
\begin{align*} h_1(a,b,c)
=~& \frac{1}{2^{\omega(c')}}
\sum_{nn'n''n''' =c'}\mu(n'')\left(\frac{a}{n}\right)
\left(\frac{b}{n'}\right) \left(\frac{ab}{n''}\right)
\\
& \times \left(\frac{\varepsilon_22^{\beta }  b' }{d_0d_{13}}\right)
\left(\frac{2^\gamma c' }{d_0d_{12}}\right)
   \left(\frac{   d_{23}  }{d_{12}d_{13}}\right)
(-1)^{\vartheta(d_0d_{12})\vartheta(d_0d_{13}) }.
\end{align*}
Puis, prenant $\tilde u$ comme
  dans l'\'enonc\'e du lemme,
nous obtenons
le r\'esultat attendu
apr\`es calcul de $ (\frac{   n n'
}{d_0d_{12}})(\frac{d_0d_{12}}{   n n'  })
$
par la loi de r\'eciprocit\'e quadratique \eqref{loiquad}.
\end{proof}

\begin{lemma}\label{lem:h2}
Lorsque  $(a,b,c)\in S$, $f(a,b)\neq 0$ et $$
(a,b,c)=(2^\alpha u,2^\beta v,2^\gamma w),
$$
avec $u,v,w$ impair, on a
$$
h_2 (a, b,c)=
\begin{cases}
1, & \mbox{si $a\equiv 1\bmod 8$},\\
1, & \mbox{si $2\mid a$ et $b\equiv
1\bmod 8$},\\
(-1)^{\vartheta(v)\vartheta(w)+ \frac{\gamma(v^2-1)}{8}+\frac{w^2-1}{8}},
& \mbox{si $2\mid (a,b)$ et $uv\equiv 1\8$},\\
(-1)^{\vartheta(b)\vartheta(w)+ \frac{\gamma(b^2-1)}{8}},
& \mbox{si $a\= 3,5,7\8$ et $1\in \{b, ab\} \8$},\\
0, & \mbox{sinon}.
\end{cases}
$$
\end{lemma}

\begin{proof}
Si $a\equiv 1\bmod 8$ alors $a\in {\QQ_2^*}^2$ et ainsi $h_2(a,b,c)=1$.
Si $2\mid a$ et $2\nmid b$, alors $f(a,b)\neq 0$ implique $b\equiv
1\bmod 8$. La formule \eqref{eq:hilbert-2} implique encore
$h_2(a,b,c)=1.$
Si $2\mid (a,b)$, alors $f(a,b)\neq 0$ implique $ab\equiv 4\bmod
{32}$. Donc la
formule \eqref{eq:hilbert-2} implique le r\'esultat.
Si $a\equiv  3,5,7\bmod 8$, alors $f(a,b)\neq 0$ implique que $b$ est
impair et que  $b$ ou  $ab$  est congru \`a $1\bmod 8$. Lorsque $c=2^\gamma w$,
la formule \eqref{eq:hilbert-2} implique le r\'esultat.
\end{proof}

Il est clair que la valeur de $h_2(a,b,c) $ ne
d\'epend que de la valeur modulo $8$ de $(u,v,w)$ et des valuations 
$2$-adiques $\alpha,\beta,\gamma$.

Avec les notations \eqref{parametre}, nous avons
$$a/(a,b)=\ve_12^{\alpha-\min\{\alpha,\beta \}}a'd_{13},\quad
b/(a,b)=\ve_2 2^{\beta-\min\{\alpha,\beta \}}b'd_{23},\qquad (a,b)=2^{
\min\{\alpha,\beta \}} d_0d_{12}.
$$
Dans la sommation \eqref{calculfab}, nous
rempla\c{c}ons
$(k,k',\ell,\ell',m,m')$ par
$$
(kk_{13},k'k'_{13},\ell\ell_{23},\ell' \ell'_{23},m_0m_{12},m'_0m'_{12}),
$$
tels que
\begin{equation}\label{condi}
k k'=a',\,k_{13}k_{13}'=d_{13},\quad \ell \ell'=b',\,\ell_{23}
\ell'_{23}=d_{23},\qquad  m_0m'_0=d_0,\, m_{12} m'_{12}=d_{12}.
\end{equation}
Lorsque $(a2^{-\alpha},b2^{-\beta})\equiv (a_0,b_0)\bmod 8 $ avec
$(a_0,b_0)\in E_{\alpha,\beta}$, $\varepsilon_1a>0, $ $
\varepsilon_2b>0$ o\`u $\bve\in \{\pm 1\}^2$,
la formule \eqref{calculfab} s'\'ecrit aussi
\begin{align*}
    f(a,b)=~&
   \frac{ 1 }{2^{\omega(2^{-\alpha-\beta}ab)}}
   \sum
\left(\frac{\varepsilon_22^\beta\ell
\ell'\ell_{23}\ell_{23}'m_0m_0'm_{12}m_{12}'}{kk_{13}}\right)
\\
& \times  \left(\frac{\varepsilon_1 2^{\alpha}
kk'k_{13}k_{13}'m_0m_0'm_{12}m_{12}'}{\ell\ell_{23}}\right)\left(\frac{\varepsilon_1\varepsilon_22^{\alpha+\beta} 
kk'k_{13}k_{13}' \ell \ell'\ell_{23}\ell_{23}'}{m_0m_{12}}\right).
\end{align*}
avec $a=\varepsilon_12^\alpha kk'k_{13}k_{13}'m_0m_0'm_{12}m_{12}'  $
et $b=\varepsilon_22^\beta \ell
\ell'\ell_{23}\ell_{23}'m_0m_0'm_{12}m_{12}' $
satisfaisant \eqref{condi}.
La formule \eqref{appliloi} fournit alors
\begin{align*}
    f(a,b)
    =~&
\frac{ 1 }{2^{\omega(2^{-\alpha-\beta}ab)}}  \sum
u(kk_{13},\ell\ell_{23},m_0m_{12})
v(k,\ell,k_{13},k_{13}',\ell_{23},\ell_{23}',m_0m_{12},m_0'm_{12}')
\\
&\times \left(\frac{  \ell'  }{k  }\right)\left(\frac{ k'  }{\ell
}\right)\left(\frac{k  }{  \ell_{23}'m_0'm_{12}' }\right)\left(\frac{
k'  }{ \ell_{23}   m_0m_{12}}\right)\left(\frac{\ell}{  k_{13}'m_0'
m_{12}'}\right)
\left(\frac{  \ell'  }{ k_{13}m_0m_{12} }\right),
\end{align*}
avec
\begin{align*}
v
&=v(k,\ell,k_{13},k_{13}',\ell_{23},\ell_{23}',m_0m_{12},m_0'm_{12}') \\
&=
\left(\frac{   \ell_{23}' }{ k_{13} m_0m_{12}}\right)\left(\frac{
k_{13}' }{ \ell_{23}    m_0m_{12}}\right)\left(\frac{m_0'm_{12}'}{
k_{13} \ell_{23} }\right)
(-1)^{\vartheta(k_{13}'m_0'm_{12}')\vartheta(\ell)+\vartheta(\ell_{23}'m_0'm_{12}')\vartheta(k)}.
\end{align*}
Le param\'etrage \eqref{condi} fournit aussi
\begin{align*} h_1(a,b,c) =&
\frac{1}{2^{\omega(c')}}
\sum_{nn'n''n''' =c'}\tilde u~ \left(\frac{ kk'  }{nn''}\right)
\left(\frac{ \ell\ell'  }{n'n''m_0m_0'k_{13}k_{13}'}\right),
\end{align*}
gr\^ace au Lemme \ref{lem:h1}.

Lorsque $(a2^{-\alpha},b2^{-\beta},c2^{-\gamma})\equiv
(u_0,v_0,w_0)\bmod 8 $ avec $(u_0,v_0)\in E_{\alpha,\beta}$,
$\varepsilon_1a>0, $ $ \varepsilon_2b>0$ o\`u
$\bve\in \{\pm 1\}^2$,   nous devons donc sommer
le terme
\begin{align*}
f(a,b) h(a,b,c) =~&
\sum\frac{\tilde u_1}{2^{\omega(2^{-\alpha-\beta-\gamma}abc) }} 
\left(\frac{ kk'
}{nn''}\right) \left(\frac{ \ell\ell'
}{n'n''m_0m_0'k_{13}k_{13}'}\right) \left(\frac{  \ell'  }{k 
}\right)\left(\frac{ k'  }{\ell   }\right)
\\
&\times \left(\frac{k  }{  \ell_{23}'m_0'm_{12}' }\right)\left(\frac{
k'  }{ \ell_{23}   m_0m_{12}}\right)\left(\frac{\ell}{  k_{13}'m_0'
m_{12}'}\right)
\left(\frac{  \ell'  }{ k_{13}m_0m_{12} }\right)\\
=~&
\sum\frac{\tilde u_1}{2^{\omega(2^{-\alpha-\beta-\gamma}abc)}} \left(\frac{ kk'
}{nn''}\right) \left(\frac{ \ell\ell'  }{n'n'' }\right) \left(\frac{ 
\ell'  }{k  }\right)\left(\frac{ k'  }{\ell   }\right)
\\& \times
\left(\frac{k  }{  \ell_{23}'m_0'm_{12}' }\right)\left(\frac{ k'  }{
\ell_{23}   m_0m_{12}}\right)\left(\frac{\ell}{  k_{13}m_0
m_{12}'}\right)
\left(\frac{  \ell'  }{ k_{13}'m_0'm_{12} }\right)
,
\end{align*}
avec des sommations sur les entiers satisfaisant \eqref{condi} et
$c'=nn'n''n'''$,
o\`u
\begin{align*}
\tilde u_1 =~&
\tilde u_1(k ,k',\ell,\ell'
,k_{13},k_{13}',\ell_{23},\ell_{23}',m_0,m_0',m_{12},m_{12}')
  \\
  =~&
u(kk_{13},\ell\ell_{23},m_0m_{12})v(k,\ell,k_{13},k_{13}',\ell_{23},\ell_{23}',m_0m_{12},m_0'm_{12}')\\
&\times
\tilde
u(n,n',n'',n''', 
m_0m_0',m_{12}m_{12}',k_{13}k_{13}',\ell_{23}\ell_{23}' ) h_2(a,b,c).
\end{align*}
Ici, nous avons
\begin{align*}
a&=\varepsilon_12^\alpha
kk'k_{13}k_{13}'m_0m_0'm_{12}m_{12}', \\
b&=\varepsilon_22^\beta \ell
\ell'\ell_{23}\ell_{23}'m_0m_0'm_{12}m_{12}', \\
c&=2^\gamma
nn'n''n'''k_{13}k_{13}'\ell_{23}\ell_{23}'m_0m_0'.
\end{align*}
\goodbreak
La sommation sur $(a,b,c)$ s'op\`ere de la m\^eme mani\`ere que pour
$N_1(B)$. Quitte \`a n\'egliger une contribution englob\'ee dans le
terme d'erreur, nous pouvons supposer que
$$
m_0m_0'm_{12}m_{12}'\leq V, \quad k_{13}k_{13}'\leq V, \quad
\ell_{23}\ell_{23}'\leq V,
$$
avec $V=(\log P)^B$.
Puis en appliquant le Lemme \ref{l:filem2} du fait du facteur
$\left(\frac{  \ell'  }{k  }\right)\left(\frac{ k'  }{\ell
}\right)$, nous pouvons nous restreindre au  cas $k ,\ell \leq V $ ou
$k',\ell'\leq V $.  De m\^eme, gr\^ace au Lemme \ref{l:filem2} et le facteur
$(\frac{ kk'
}{nn''}) (\frac{ \ell\ell'  }{n'n'' })$,
nous pouvons d\'esormais supposer que $n,n',n''\leq V$.

  Le facteur $\tilde u_1$ ne d\'epend pas de
$k,\ell,k',\ell'$ mais seulement de leur valeur modulo $8$.
En fixant $(k',\ell')\equiv (k'_0,\ell'_0)\bmod 8$, la somme sur 
$(k',\ell')$ lorsque $k ,\ell \leq V $
\`a estimer est
\begin{equation}\label{Q1}
Q\left( \frac{P}{2^{\alpha} m_0m_0'm_{12}m'_{12}kk_{13}k'_{13} } , 
\frac{P}{2^{\beta}
m_0m_0'm_{12}m'_{12}\ell \ell_{23}\ell'_{23} }
;k_0',\ell_0',\d,8,\q,\chi_{q_1},\chi_{q_2}\right),
\end{equation}
avec
\begin{align*}
q_1&= \ell_{23}   m_0m_{12} nn''\ell ,\qquad\qquad\quad
q_2=k_{13}'m_0'm_{12}    n'n''k,\\
d_1&= n' n''' m_0' m'_{12}kk_{13}k'_{13} \ell'_{23},\qquad d_2= nn'''
m_0m'_{12} k_{13}  \ell\ell_{23}\ell'_{23}.
\end{align*}
Respectivement, en fixant   $(k ,\ell )\equiv (k_0,\ell_0)\bmod 8$,
la somme sur  $(k,\ell)$  \`a estimer lorsque $k',\ell'\leq V $ est
\begin{equation}\label{Q2}
Q\left( \frac{P}{2^{\alpha} m_0m_0'm_{12}m'_{12}k'k_{13}k'_{13} 
},\frac{P}{2^{\beta}
m_0m_0'm_{12}m'_{12}\ell' \ell_{23}\ell'_{23} }
;k_0,\ell_0,\d',8,\q',\chi_{q_1'},\chi_{q_2'}\right),
\end{equation}
avec
\begin{align*}
q_1'&=m_0'm_{12}'   nn''\ell'\ell_{23}' ,\qquad\qquad\quad
q_2'=k_{13}m_0m_{12}'    n'n''k',
\\ d_1'&= n' n''' m_0 m_{12}k'k_{13}k'_{13} \ell_{23},\qquad d_2'=
nn''' m_0'm_{12} k_{13}'  \ell\ell_{23}\ell'_{23}.
\end{align*}

La contribution principale provient du cas o\`u les deux modules
$q_1,$ $q_2$ (resp.\ $q_1',$ $q_2'$) des caract\`eres
somm\'es sont \'egaux \`a un. Apr\`es cette sommation, il reste
quatre variables \`a sommer $(n''',k_{13},\ell'_{23}, m'_{12})$
(resp.\     $(n''',k'_{13},\ell_{23}, m_{12})$).

\medskip

Compte tenu du Corollaire \ref{c:estQ}, dans le premier cas,
le terme principal obtenu pour l'\'evaluation de \eqref {Q1} est
\begin{align*}
{\bf 1}_{\substack{
\ell_{23}   m_0m_{12} nn''\ell=1\\
k_{13}'m_0'm_{12}    n'n''k=1}}
\frac{ \tilde u_2(m_{12}')}{\pi^3}
   &\frac{\mu^2(2 k_{13} \ell'_{23})}{2^{2+\alpha+\beta}}
     \frac{ \phi_2(  k_{13} \ell'_{23}  m'_{12} )}{  {m'_{12}}^{\! \!
\! 2}\ell_{23}'k_{13} }
   \frac{\mu^2(n''')\phi_2( n'''  )}{2^{\omega(n'''\ell_{23}'k_{13} m'_{12} )}}
\left(\frac{ n'''  }{m'_{12}  }\right) \\
&\times\frac{P^2}{\log
P}\left\{ 1+O\left(\frac{(\log 2\ell_{23}'k_{13} m'_{12} )^{\frac{3}{2}}}{\log
P}\right)\right\},
\end{align*}
avec
\begin{align*}\tilde u_2(m_{12}')&=
(-1)^{\vartheta( m_{12}')\vartheta( k_{13})}  \left(\frac{m'_{12}}{
k_{13}}  \right)
   \left(\frac{2^\gamma \ell_{23}' }{ m_{12}' }\right)
h_2(\varepsilon_12^\alpha  k'_0k_{13} m_{12}',\varepsilon_22^\beta
\ell'_0 \ell_{23}' m_{12}', 2^\gamma n'''k_{13} \ell_{23}' ).
\end{align*}
Nous avons utilis\'e ici la relation  $u( k_{13},1,1)
  (\frac{\varepsilon_22^{\beta }   }{ k_{13} })=1,$
issue de \eqref{defu}.

Ensuite nous sommons les $n'''\leq P/(2^{\gamma}k_{13}\ell_{23}')$
congrus \`a $n'''_0\bmod 8$ avec $n_0'''$ impair, premiers \`a
$2k_{13}\ell_{23}'$ en appliquant le Lemme \ref{l:ficor2} avec
$\nu=2$. Nous obtenons un terme principal
\begin{align*} {\bf 1}_{ m_{12}' =1}  \frac{ \tilde u_2(1)}{\pi^3}
   \frac{\mu^2(2 k_{13} \ell'_{23})}{2^{4+\alpha+\beta+\gamma}}
\frac{ \phi_2(k_{13}
\ell'_{23})c_2(2   k_{13} \ell'_{23} )
}{2^{\omega( k_{13} \ell'_{23} )} (k_{13} \ell'_{23})^2}
\frac{P^3}{(\log P)^{\frac{3}{2}}}\left\{ 1+O\left(\frac{(\log
2\ell_{23}'k_{13}  )^{\frac{3}{2}}}{\log P}\right)\right\},
\end{align*}
avec
$
\tilde u_2(1)=
h_2(\varepsilon_12^\alpha  k'_0k_{13}  ,\varepsilon_22^\beta
\ell'_0 \ell_{23}'  ,2^\gamma n'''_0k_{13} \ell_{23}' ).
$
Nous aurons besoin du lemme suivant.

\begin{lemma}\label{lem:last}
Soient $z\in \{\pm 1\}$ et
$$
\nu_2(z)=\sum_{(\alpha,\beta,\gamma)\in \{ 0,1\}^3}\sum_{\substack{
(u_0  , v_0   )\in E_{\alpha,\beta}\bmod 8\\ w_0 \in
(\ZZ/8\ZZ)^*}}\frac{ 
z^{\vartheta(u_0)\vartheta(v_0)}}{2^{\alpha+\beta+\gamma}}h_2( 
2^\alpha
u_0,2^\beta   v_0    ,2^\gamma w_0) .
$$
Alors $\nu_2(z)=68$.
\end{lemma}

\begin{proof}
Nous utilisons l'expression \eqref{eq:set-E} pour les $E_{\alpha,\beta}$ et
le Lemme \ref{lem:h2} pour le calcul de $h_2$.
La  somme $\nu_2(z)$
se d\'ecompose sous la forme suivante
$$
=\begin{cases}
36,  & \mbox{si $\alpha=0$, $u_0\equiv 1\bmod 8$},\\
12, & \mbox{si $(\alpha,\beta)=(1,0)$, $v_0\equiv 1\bmod 8$}, \\
0,& \mbox{si $(\alpha,\beta)=(1,1)$},  \\
16,&  \mbox{si $(\alpha,\gamma)=(0,0)$, $u_0\equiv 3,5,7\bmod 8$},  \\
4,& \mbox{si $(\alpha,\gamma)=(0,1)$, $u_0\equiv 3,5,7\bmod 8$}.
\end{cases}
$$
Pour
  la troisi\`eme ligne,
on a not\'e que
$$
\sum_{w_0\in (\ZZ/8\ZZ)^*} z^{\vartheta(w_0)}(-1)^{\frac{w_0^2-1}{8}}=0,
$$
ce qui ach\`eve la d\'emonstration.
\end{proof}

Ensuite, avec la notation de
\S \ref{s:lemmes}, nous  avons
\begin{align*}
c_2(2r)
&=
c_2(1)\prod_{p\mid 2r } \Big(1+\frac{1}{2(p+1)}\Big)^{-1} \\
&=
\frac{6}{7\sqrt{\pi}}\prod_{p }
\Big(1+\frac{1}{2(p+1)}\Big)\Big(1-\frac{1}{p}\Big)^{\frac{1}{2}}
\prod_{p\mid r } \Big(1+\frac{1}{2(p+1)}\Big)^{-1},
\end{align*}
pour un impair $r$. Du Lemme \ref{lem:last},
il  d\'ecoule la formule
\begin{align*}
\sum_{(\alpha,\beta,\gamma)\in \{ 0,1\}^3}
\sum_{\substack{\varepsilon_1,\varepsilon_2\in \{\pm 1\}\\
(\varepsilon_1  k'_0k_{13}  ,\varepsilon_2 \ell'_0 \ell_{23}' )\in
E_{\alpha,\beta}\bmod 8\\ n_0'''\in (\ZZ/8\ZZ)^*}}
\frac{ \tilde u_2(1)}{2^{4+\alpha+\beta+\gamma}}
&
= \frac{\nu_2(1)}{4}=17.
\end{align*}
Puis enfin, nous sommons sur $k_{13}$ et $ \ell'_{23}$.
Nous avons
$$
\sum_{\substack{k_{13},\ell'_{23}\in \NN
}}  \frac{\mu^2(2 k_{13} \ell'_{23})
  \phi_2(k_{13}
\ell'_{23})
}{2^{\omega( k_{13} \ell'_{23} )} (k_{13} \ell'_{23})^2}
\prod_{p\mid k_{13} \ell'_{23} } \Big(1+\frac{1}{2(p+1)}\Big)^{-1}
  =
\prod_{p> 2} \Big(1+\frac{1}{p(p+\frac{3}{2})}\Big).
$$
   Nous sommes donc  amen\'es  \`a une premi\`ere contribution de 
\eqref{Q1} \`a $N_2(P)$   \'egale \`a
\begin{equation}\label{eq:soap}
17\times \frac{6}{7\pi^{\frac{7}{2}}} \times \frac{7}{8}
\prod_p
   \frac{(1- \frac{1}{ p
})^{\frac{1}{2}}}{(1+ \frac{1}{ p })}
\Big( 1+\frac{3}{ 2p }+\frac{1}{  p^2 }\Big)\frac{P^3}{(\log
P)^{\frac{3}{2}}}\left\{ 1+O\left(\frac{1}{\log P}\right)\right\}.
\end{equation}

\medskip

Nous passons maintenant \`a la deuxi\`eme contribution, li\'ee \`a \eqref{Q2}.
Gr\^ace au Corollaire~\ref{c:estQ}, lorsque
$(k,\ell)\equiv
(k_0,\ell_0)\bmod 8$,
le terme principal obtenu est
\begin{align*}
{\bf 1}_{\substack{
m_0'm_{12}'   nn''\ell'\ell_{23}' =1\\
k_{13}m_0m_{12}'    n'n''k'=1
}}
\frac{\tilde u_3(m_{12} )}{\pi^3} &\frac{\mu^2(2
k_{13}' \ell_{23})}{
2^{2+\alpha+\beta}}
\frac{\phi_2(k_{13}' \ell_{23}  m_{12} )}{m_{12}^2\ell_{23} k_{13}' }
   \frac{\mu^2(n''')\phi_2(  n''' )}{2^{\omega(n'''\ell_{23} k_{13}' m_{12} )}}
\left(\frac{ n'''  }{m_{12}  }\right) \\
&\times\frac{P^2}{\log
P}\left\{ 1+O\left(\frac{(\log 2\ell_{23}k_{13}' m_{12} )^{\frac{3}{2}}}{\log
P}\right)\right\},
\end{align*}
avec
\begin{align*}
\tilde u_3(m_{12})
=~&(-1)^{\vartheta(\ell_0\ell_{23} m_{12})\vartheta(k_{13}' )  }
u(k_0,\ell_0\ell_{23}, m_{12}) \left(\frac{\varepsilon_22^{\beta }
}{ k_{13}' }\right)
   \left(\frac{2^\gamma \ell_{23}k_{13}' }{ m_{12}}\right)\\
   &\times
h_2(\varepsilon_12^\alpha  k _0k_{13}' m_{12} ,\varepsilon_22^\beta
\ell_0 \ell_{23} m_{12}, 2^\gamma n'''k_{13}'\ell_{23}).  \end{align*}
   Ensuite, nous sommons les $n'''\leq P/(2^{\gamma}k_{13}'\ell_{23} )$
congrus \`a $n'''_0\bmod 8$ avec $n_0'''$ impair, premiers \`a
$2k_{13}'\ell_{23}' $ en appliquant le Lemme \ref{l:ficor2} avec
$\nu=2$. Nous obtenons un terme principal
\begin{align*}
  {\bf 1}_{ m_{12}  =1}   \frac{ \tilde u_3(1)}{\pi^3}
   \frac{\mu^2(2 k_{13}' \ell_{23})}{2^{
  4+\alpha+\beta+\gamma}}
\frac{    \phi_2(k_{13}' \ell_{23})c_2(2
k_{13}' \ell_{23} )}
{2^{\omega( k_{13}' \ell_{23}
)} (k_{13}' \ell_{23})^2}
\frac{P^3}{(\log
P)^{\frac{3}{2}}}\left\{ 1+O\left(\frac{(\log 
2\ell_{23}k_{13}')^{\frac{3}{2}}}{\log
P}\right)\right\}.
\end{align*}
Nous avons
\begin{align*}\tilde u_3(1)&= (-1)^{\vartheta(\ell_0\ell_{23})
\vartheta( k_{13}')}u(k_0,\ell_0\ell_{23}, 1)
       \left(\frac{\varepsilon_22^{\beta }   }{  k_{13}' }\right)
h_2(\varepsilon_12^\alpha  k _0k_{13}'   ,\varepsilon_22^\beta
\ell_0 \ell_{23}  , 2^\gamma n'''_0k_{13}'\ell_{23})
\\ &=  u(k_0k_{13}',\ell_0\ell_{23}, 1)
h_2(\varepsilon_12^\alpha  k _0k_{13}'   ,\varepsilon_22^\beta
\ell_0 \ell_{23}  , 2^\gamma n'''_0k_{13}'\ell_{23}).  \end{align*}
o\`u nous avons utilis\'e la d\'efinition \eqref{defu} de $u$.
D'apr\`es les Lemmes \ref{lem:calculu} et \ref{lem:last}, nous avons
\begin{align*}
\sum_{(\alpha,\beta,\gamma)\in \{ 0,1\}^3}
&\sum_{\substack{\varepsilon_1,\varepsilon_2\in \{\pm 1\}\\
(\varepsilon_1 k _0k_{13}'  ,\varepsilon_2    \ell_0 \ell_{23} )\in
E_{\alpha,\beta}\bmod 8\\ n_0'''\in (\ZZ/8\ZZ)^*}}
\frac{ \tilde u_3(1)}{2^{4+\alpha+\beta+\gamma}}\\
& =
   \sum_{(\alpha,\beta,\gamma)\in \{ 0,1\}^3}
\!\!\!\! \sum_{\substack{\varepsilon_1,\varepsilon_2\in \{\pm 1\}\\
(u_0  , v_0   )\in E_{\alpha,\beta}\bmod 8\\ w_0\in (\ZZ/8\ZZ)^*}}
\!\!\!\!\!\!(-1)^{\vartheta(\varepsilon_1 )\vartheta(\varepsilon_2
)}\frac{ (-1)^{\vartheta(u_0)\vartheta(v_0)
}}{2^{4+\alpha+\beta+\gamma}}
h_2( 2^\alpha u_0  , 2^\beta  v_0, 2^\gamma w_0)
\\&= \frac{\nu_2(-1)}{8}=\frac{17}{2}.
\end{align*}
Puis enfin une sommation sur $k_{13}$ et $ \ell'_{23}$
fournit une deuxi\`eme contribution \`a $N_2(P)$   \'egale \`a la 
moiti\'e de \eqref{eq:soap}.

\smallskip

Pour conclure, nous avons montr\'e la formule
$$
N_2(P)=\frac{153}{ 8\pi^{\frac{7}{2}}} \prod_{p}
   \frac{(1- \frac{1}{ p
})^{\frac{1}{2}}}{(1+ \frac{1}{ p })}
\Big( 1+\frac{3}{ 2p }+\frac{1}{  p^2 }\Big)
\frac{P^3}{(\log
P)^{\frac{3}{2}}}\left\{ 1+O\left(\frac{1}{\log P}\right)\right\}   .$$
En combinant, dans \eqref{eq:NBr},
cette estimation avec
\eqref{estN1P},
nous achevons la d\'emonstration du  Th\'eor\`eme \ref{main}.

\goodbreak

\end{document}